\documentclass[pdflatex,sn-mathphys-num]{sn-jnl}
\usepackage{amsmath,amssymb,amsthm,amsfonts}
\usepackage{booktabs}
\usepackage{multirow}
\usepackage{array}
\usepackage{graphicx}
\usepackage{enumitem}
\usepackage{placeins}
\usepackage{bm}
\usepackage{flafter}

\theoremstyle{thmstyleone}

\newtheorem{proposition}{Proposition}

\theoremstyle{thmstyletwo}

\theoremstyle{thmstylethree}

\begin{document}

\title[Sparse Corner-Difference MILP for Minimum Square Tiling]{A Sparse Corner-Difference MILP for Minimum Square Tiling}

\author[1,2]{\fnm{Zhuo} \sur{Yu}}
\author*[1,2]{\fnm{Yuan} \sur{Wang}}\email{wangyuan@cuhk.edu.cn}
\author[1]{\fnm{Zhuo} \sur{Liu}}

\affil[1]{\orgdiv{\orgname{Shenzhen Research Institute of Big Data}}, \orgaddress{\city{Shenzhen}, \country{P.R. China}}}

\affil[2]{\orgdiv{\orgname{The Chinese University of Hong Kong}}, \orgaddress{\city{Shenzhen}, \country{P.R. China}}}

\abstract{
We study the problem of tiling an $N\times N$ square with axis-aligned squares whose side lengths are integers strictly less than $N$, with the objective of minimizing the number of tiles.
The standard placement-based exact-cover MILP contains
$\Theta(N^5)$ placement-to-cell nonzero coefficients and is also
affected by the dihedral symmetry of the square domain. We
introduce a Corner-Difference MILP (CD-MILP) that represents each
selected square by at most four signed corner coefficients.
Two-dimensional prefix reconstruction shows that the resulting
constraints are equivalent to the original cell-cover equations,
while reducing the coverage-related nonzero count to
$\Theta(N^3)$. We also introduce lightweight corner-ordering
constraints that retain at least one representative from every
$D_4$ orbit, although ties may leave residual symmetry.
Computational experiments on nine prime-size instances compare
the baseline, CD-MILP, the baseline with corner ordering, and
their combination using Gurobi and COPT. Under the stated
experimental settings, the combined formulation solves eight
instances with Gurobi and seven with COPT, compared with four and
five, respectively, for the baseline. Its average presolved
matrixs contain approximately $10^5$ nonzero coefficients on average,
compared with approximately $10^7$ for the baseline.
}



\maketitle

\section{Introduction}
\label{sec:intro}

We consider the problem of tiling an \(N\times N\) square with smaller, axis-aligned squares of integer side lengths, aiming to minimize the total number of tiles. Throughout the paper, every tile side length is required to be strictly smaller than \(N\), and repeated side lengths are allowed. This classical geometric covering problem is related to the literature on squared rectangles and ``squaring the square'', but the minimum tiling variant studied here has a different optimization objective because it asks for the smallest number of squares rather than a dissection with distinct sizes. It is also a representative exact-cover problem in geometric integer programming: local placement decisions must collectively satisfy a global cell-by-cell coverage condition.

The study of squared rectangles dates back to Dehn \cite{dehn1903} and the famous work of Brooks, Smith, Stone, and Tutte \cite{brooks1940}, who established a connection between perfect squared rectangles and electrical networks. Subsequently, Kurz \cite{kurz2012} introduced integer linear programming (ILP) models to compute optimal tilings for moderate sizes. More recently, Monaci and dos Santos \cite{monaci2018} developed a heuristic algorithm for the minimum tiling of rectangles by squares, and the MIT CompGeom Group \cite{mit2023} characterized when a rectangle can be tiled using only squares of side at least~2. These works place square tiling at the intersection of discrete geometry, exact-cover modeling, and computational optimization.

Kurz \cite{kurz2012} proposed an ILP formulation that uses binary placement variables and exact-cover constraints that require each unit cell to be covered exactly once. This natural formulation has \(\Theta(N^3)\) candidate placements and \(\Theta(N^2)\) cell-cover constraints, but the placement-to-cell incidence matrix contains \(\Theta(N^5)\) nonzero entries. The difficulty of the standard formulation is therefore not merely the number of candidate placements. Rather, the incidence matrix explicitly records every cell affected by every placement, so large squares create dense columns and expensive LP relaxations. In addition, the dihedral symmetry of the square creates equivalent solutions that can lead to a redundant branch-and-bound search.

Our key observation is that the two-dimensional discrete derivative of an axis-aligned square indicator is supported only on its four corners. While an \(h\times h\) square covers \(h^2\) cells, its two-dimensional difference has only four nonzero corner events. This suggests replacing dense cell incidence with signed local updates and reconstructing coverage through prefix sums.

We introduce a Corner-Difference MILP (CD-MILP) for the minimum square tiling problem. The proposed model should be interpreted as a sparse equivalent reformulation: it changes the linear-algebra structure exposed to the solver, not the underlying set of placement decisions, objective, or exact-cover semantics. Since the square domain also induces \(D_4\)-equivalent tilings, we complement the sparse representation with inexpensive corner-ordering inequalities based on the side lengths of the four corner-incident squares.

\textbf{Relation to formulation and symmetry-handling literature.}
The use of cumulative or difference-based representations is a standard modeling idea in optimization and related computational settings. Compact extended formulations often replace dense global relations by auxiliary variables and sparse linking constraints \cite{lancia2017,vielma2015}; inventory-balance models provide a classical one-dimensional example in which cumulative effects are represented through local flow-balance equations \cite{pochet2006}. Prefix-sum and summed-area representations are also widely used for efficient cumulative queries in computational applications \cite{crow1984,viola2001}. Our contribution is not the difference-array idea in isolation, but its application to the placement-based exact-cover MILP for square tiling: each selected square contributes four signed corner events, and exact cell coverage is recovered by two-dimensional prefix sums.

The closest prior formulation to ours is the placement-based exact-cover ILP used in computational studies of square tiling \cite{kurz2012}. That baseline provides the placement space, objective, and exact-cover semantics. The sparse corner-difference reformulation and the corner-ordering inequalities studied here are formulation-level modifications to that baseline rather than a new relaxation or a heuristic search procedure.

Symmetry handling is also a well-studied topic in integer programming, including orbitopes, orbital branching, and other symmetry-exploitation methods \cite{kaibel2008,margot2010,ostrowski2011,fischetti2017}. The corner-ordering inequalities used here should be viewed as lightweight problem-specific symmetry-breaking constraints. They do not attempt to describe a full orbitopal formulation or eliminate every symmetric representative. Instead, they use the side lengths of the four corner-incident squares to keep at least one representative from each \(D_4\) orbit while adding only \(O(N)\) nonzero coefficients. In this paper, we make three contributions targeting these two bottlenecks.

\begin{enumerate}
    \item \textbf{Corner-Difference MILP.} We replace the dense cell-incidence constraints by four signed corner-difference events and a prefix-sum reconstruction, reducing the number of placement-related nonzeros from \(\Theta(N^5)\) to \(\Theta(N^3)\) while preserving the feasible set in the original placement variables.
    
    \item \textbf{Corner-ordering symmetry breaking.} We add corner-incidence constraints that reduce dihedral symmetry by selecting canonical representatives of symmetric tilings, thereby pruning redundant branches without removing all representatives of any symmetry orbit.
    
    \item \textbf{Cross-solver empirical validation.} We evaluate the reformulation, the symmetry-breaking constraints, and their combination using \textsc{Gurobi} and \textsc{COPT}. On nine prime-size benchmark instances, CD-MILP with corner ordering solves eight instances with \textsc{Gurobi} and seven with \textsc{COPT}, compared with four and five for the exact-cover baseline.
\end{enumerate}

The remainder of the paper is organized as follows. Section~\ref{sec:orig} introduces the standard placement-based exact-cover MILP formulation and identifies the dense placement-to-cell incidence structure that motivates our reformulation. Section~\ref{sec:2ddiff} presents the proposed Corner-Difference MILP, including the four-corner difference representation, the two-dimensional prefix-sum reconstruction, the projected equivalence with the exact-cover model, and the resulting sparsity reduction. Section~\ref{sec:symmetry} analyzes the dihedral symmetries of the square tiling problem and introduces the corner-ordering inequalities used for lightweight symmetry breaking. Section~\ref{sec:num} reports the computational experiments with \textsc{Gurobi} and \textsc{COPT}, comparing the exact-cover baseline, CD-MILP, the symmetry-enhanced baseline, and their combination. Section~\ref{sec:conclu} concludes the paper and discusses possible extensions. Additional ablation results, model-size statistics, and representative optimal tilings are provided in the appendices.

\section{Exact-Cover MILP Formulation}
\label{sec:orig}
For reference, Table~\ref{tab:notation} summarizes the notation used throughout the paper.

\begin{table}[!htbp]
\centering
\caption{Notation}
\label{tab:notation}
\begin{tabular}{ll}
\toprule
Symbol & Description \\
\midrule
\textit{Sets and indices}\\
$N$ & side length of the large square \\
$h$ & side length of a candidate tile \\
$(i,j)$ & upper-left cell coordinate of a placement \\
$(u,v)$ & unit-cell index \\
$(p,q)$ & grid-point / difference-array index \\
$\mathcal P$ & set of feasible placements \\
\textit{Decision variables}\\
$x_{i,j,h}$ & binary placement variable \\
$d_{p,q}$ & 2D difference event variable \\
$c_{u,v}$ & reconstructed cumulative coverage \\
$\alpha^{p,q}_{i,j,h}$ & signed corner coefficient \\
$H^{\mathrm{TL}},H^{\mathrm{TR}},H^{\mathrm{BL}},H^{\mathrm{BR}}$ & corner-incident square side lengths \\
\bottomrule
\end{tabular}
\end{table}
The baseline model is the standard placement-based exact-cover MILP used in prior ILP treatments of square tiling \cite{kurz2012}. Let $N\in\mathbb Z_{\ge 2}$ denote the side length of the large square. We index unit cells by $(u,v)$ for $1\le u,v\le N$ and grid points, used later for difference events, by $(p,q)$ for $1\le p,q\le N+1$. Smaller squares have integer side lengths
$h \in \{1,2,\dots,N-1\}$. A square of size $h$ placed at upper-left coordinate $(i,j)$ must satisfy
$1 \le i \le N-h+1$, $1 \le j \le N-h+1$. Define the placement index set
\[
\mathcal P = \{(i,j,h) \mid 1\le h\le N-1,\; 1\le i,j\le N-h+1\}.
\]

For each placement $(i,j,h)\in\mathcal P$, define the binary decision variable
\[
x_{i,j,h} = 
\begin{cases}
1, & \text{if an } h\times h \text{ square is placed at }(i,j),\\
0, & \text{otherwise}.
\end{cases}
\qquad (i,j,h)\in\mathcal P.
\]

The objective is to minimize the number of placed squares:
\begin{equation}
\min \sum_{(i,j,h)\in\mathcal P} x_{i,j,h} \label{eq:orig_obj}
\end{equation}

The exact-cover constraints require each unit cell $(u,v)$ to be covered exactly once:
\begin{equation}
\sum_{\substack{(i,j,h)\in\mathcal P:\\ i\le u\le i+h-1,\; j\le v\le j+h-1}} x_{i,j,h} = 1,\qquad \forall u,v\in\{1,\dots,N\}. \label{eq:orig_cover}
\end{equation}

Finally, the placement variables are binary:
\begin{equation}
x_{i,j,h}\in\{0,1\},\qquad \forall (i,j,h)\in\mathcal P. \label{eq:orig_binary}
\end{equation}

The exact-cover formulation above is a natural baseline because it directly encodes the geometric meaning of a tiling: every selected square contributes coverage to each cell in its support, and every cell must receive a total coverage of one. Its drawback is not semantic but algebraic. The model represents the effect of each selected square by explicitly listing all cells it covers. Consequently, a square of side length $h$ contributes $h^2$ nonzero coefficients to the cell-cover constraints, and the resulting placement-to-cell incidence matrix has $\Theta(N^5)$ nonzeros. The following section keeps the same placement variables, objective, and exact-cover semantics, but replaces the dense incidence matrix with a sparse corner-difference representation followed by prefix-sum reconstruction.

\section{Corner-Difference MILP Formulation}
\label{sec:2ddiff}

The dense part of the exact-cover model comes from representing each selected square by all cells in its support. A square of side length $h$ appears in $h^2$ cell-cover constraints, even though its support is an axis-aligned rectangle determined by four grid points. The formulation below exploits the fact that the two-dimensional difference of a square indicator has only four corner events. It keeps the original placement decisions, but replaces dense placement-to-cell incidence by sparse four-corner updates followed by a two-dimensional prefix reconstruction.

\subsection{The complete CD-MILP formulation}
\label{sec:reform_model}

The reformulation uses the same placement variables $x_{i,j,h}$ as the exact-cover model. It also introduces integer difference variables $d_{p,q}$ for $1\le p,q\le N+1$ and integer cumulative coverage variables $c_{u,v}$ for $0\le u,v\le N$. The boundary variables $c_{0,v}$ for $0\le v\le N$ and $c_{u,0}$ for $1\le u\le N$ are fixed to zero so that the prefix recurrence has a well-defined starting row and column.

For each placement $(i,j,h)\in\mathcal P$, define the corner coefficients
\begin{equation}
\alpha^{p,q}_{i,j,h} =
\begin{cases}
+1, & (p,q) = (i,j), \\
-1, & (p,q) = (i+h,j), \\
-1, & (p,q) = (i,j+h), \\
+1, & (p,q) = (i+h,j+h), \\
0, & \text{otherwise}.
\end{cases} \label{eq:alpha}
\end{equation}
We call the resulting formulation the \textbf{Corner-Difference MILP (CD-MILP)}, because each selected square is represented by four signed corner-difference events rather than by all cells in its support:
\begin{align}
\min \quad & \sum_{(i,j,h)\in\mathcal P} x_{i,j,h} \label{eq:reform_obj}\\
\text{s.t.}\quad
& d_{p,q} = \sum_{(i,j,h)\in\mathcal P} \alpha^{p,q}_{i,j,h} x_{i,j,h},
&& 1\le p,q\le N+1, \label{eq:reform_d}\\
& c_{u,v} = c_{u-1,v}+c_{u,v-1}-c_{u-1,v-1}+d_{u,v},
&& 1\le u,v\le N, \label{eq:reform_prefix}\\
& c_{0,v}=0,
&& 0\le v\le N, \label{eq:reform_boundary_row}\\
& c_{u,0}=0,
&& 1\le u\le N, \label{eq:reform_boundary_col}\\
& c_{u,v}=1,
&& 1\le u,v\le N, \label{eq:c_cover}\\
& d_{p,q}\in\mathbb Z,\quad c_{u,v}\in\mathbb Z,
&& 1\le p,q\le N+1,\; 0\le u,v\le N, \label{eq:reform_integer}\\
& x_{i,j,h}\in\{0,1\},
&& (i,j,h)\in\mathcal P. \label{eq:reform_binary}
\end{align}

The objective~\eqref{eq:reform_obj} and the binary restrictions~\eqref{eq:reform_binary} are unchanged from the exact-cover formulation. Constraints~\eqref{eq:reform_d} define the difference variables from the selected placements. Constraints~\eqref{eq:reform_prefix}, together with the fixed boundary values~\eqref{eq:reform_boundary_row}--\eqref{eq:reform_boundary_col}, reconstruct cell coverage by prefix sums. Constraints~\eqref{eq:c_cover} then impose exact coverage on the reconstructed variables. Difference events with $p=N+1$ or $q=N+1$ are kept in~\eqref{eq:reform_d} to complete the four-corner representation, but they do not enter the cell-domain prefix recurrence~\eqref{eq:reform_prefix} and therefore do not directly affect any coverage variable. The variables $d_{p,q}$ and $c_{u,v}$ are auxiliary; all tiling decisions remain in $x$. Sections~\ref{sec:diff_representation} and~\ref{sec:prefix_reconstruction} explain the two building blocks, and Section~\ref{sec:reform_relation} proves equivalence and analyzes sparsity.

\subsection{Four-corner difference representation}
\label{sec:diff_representation}

Figure~\ref{fig:four_corner_difference} illustrates the four-corner update used in~\eqref{eq:reform_d}. The shaded region is the $h\times h$ support of one selected placement. In the exact-cover formulation, the corresponding variable appears in all $h^2$ cell-cover equations. In CD-MILP, the same placement contributes only four signed events: $+1$ at $(i,j)$ and $(i+h,j+h)$, and $-1$ at $(i+h,j)$ and $(i,j+h)$. The side length $h$ changes the locations of these events but not their number.

\begin{figure}[htbp]
\centering
\includegraphics[width=0.70\linewidth]{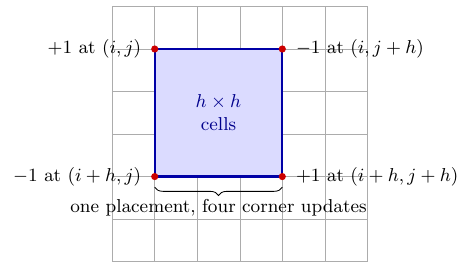}
\caption{Four-corner difference representation of a selected square.}
\label{fig:four_corner_difference}
\end{figure}

More explicitly, a placement $(i,j,h)$ covers the cell set
\[
\{i,\ldots,i+h-1\}\times \{j,\ldots,j+h-1\}.
\]
The four grid points in~\eqref{eq:alpha} are exactly the corners of this covered cell block.

\subsection{2D prefix-sum reconstruction of coverage}
\label{sec:prefix_reconstruction}

The prefix equations invert the corner-difference representation. With $c_{0,v}=c_{u,0}=0$, recurrence~\eqref{eq:reform_prefix} implies
\[
c_{u,v}=\sum_{p=1}^{u}\sum_{q=1}^{v} d_{p,q},
\qquad 1\le u,v\le N.
\]
The inclusion-exclusion term $-c_{u-1,v-1}$ in~\eqref{eq:reform_prefix} is what makes this a two-dimensional cumulative sum rather than two independent one-dimensional recurrences.

For a single placement, the cumulative sum equals one exactly on the shaded cells in Figure~\ref{fig:four_corner_difference} and zero outside them. The event at $(i,j)$ turns coverage on; the two negative events remove it once the horizontal or vertical boundary has been crossed; and the event at $(i+h,j+h)$ corrects the double subtraction beyond both boundaries. For multiple placements, linearity adds the reconstructed indicators. Constraint~\eqref{eq:c_cover} then requires the sum of these indicators to equal one on every cell, exactly as in the original exact-cover model.

\subsection{Equivalence, relaxation, and sparsity}
\label{sec:reform_relation}

\begin{proposition}[Projected equivalence]
\label{prop:equivalence}
CD-MILP~\eqref{eq:reform_obj}--\eqref{eq:reform_binary} and the exact-cover formulation~\eqref{eq:orig_obj}--\eqref{eq:orig_binary} have the same feasible set in the original placement variables $x$.
\end{proposition}

\begin{proof}
Fix any placement vector $x$. For a single placement $(i,j,h)$, summing the four coefficients in~\eqref{eq:alpha} over the rectangle $\{1,\ldots,u\}\times\{1,\ldots,v\}$ gives
\[
\sum_{p=1}^{u}\sum_{q=1}^{v}\alpha^{p,q}_{i,j,h}
=
\begin{cases}
1, & i\le u\le i+h-1,\; j\le v\le j+h-1,\\
0, & \text{otherwise}.
\end{cases}
\]
By linearity, the prefix value $c_{u,v}$ generated by~\eqref{eq:reform_d} and~\eqref{eq:reform_prefix} is exactly the number of selected squares covering cell $(u,v)$. Hence $c_{u,v}=1$ for all cells if and only if the original exact-cover constraints~\eqref{eq:orig_cover} hold. Since both models use the same objective in $x$, they have the same feasible placement vectors and the same optimal objective value.
\end{proof}

\begin{proposition}[LP relaxation]
\label{prop:lp}
The continuous relaxations of CD-MILP and the exact-cover formulation have the same projection onto the placement variables $x$.
\end{proposition}

\begin{proof}
Replace the binary restrictions on $x$ by $0\le x_{i,j,h}\le 1$ in both formulations, and relax the auxiliary integrality restrictions~\eqref{eq:reform_integer}. The argument in Proposition~\ref{prop:equivalence} is linear and does not use integrality of $x$. For any fractional placement vector satisfying the exact-cover equations, the values of $d$ and $c$ defined by~\eqref{eq:reform_d} and~\eqref{eq:reform_prefix} satisfy CD-MILP. Conversely, any fractional CD-MILP solution reconstructs exactly the cell coverages induced by its placement vector, so $c_{u,v}=1$ implies the exact-cover equations. Thus the two LP relaxations project to the same set of placement vectors.
\end{proof}

\begin{proposition}[Sparsity]
\label{prop:nnz}
CD-MILP reduces the order of the coverage-related nonzero coefficients from $\Theta(N^5)$ in the cell-incidence representation to $\Theta(N^3)$. This reduction is obtained without changing the feasible set in the placement variables.
\end{proposition}

\begin{proof}
Let $A_{\mathrm{cell}}$ denote the placement-to-cell incidence matrix of the exact-cover formulation. Its number of nonzero coefficients is

\[
\operatorname{nnz}(A_{\mathrm{cell}})
=\sum_{h=1}^{N-1}h^2(N-h+1)^2
=\Theta(N^5).
\]
The sparsity orders of constraints~\eqref{eq:reform_d}--\eqref{eq:c_cover} in CD-MILP are as follows. The difference-definition equations~\eqref{eq:reform_d} contain four coefficients for each placement variable, and therefore their $x$-variable block has $\Theta(|\mathcal P|)=\Theta(N^3)$ nonzeros. The same equations contain only $\Theta(N^2)$ additional coefficients from the auxiliary variables $d_{p,q}$. The prefix equations~\eqref{eq:reform_prefix} involve only a constant number of neighboring $c$ and $d$ variables per cell, so they contribute $\Theta(N^2)$ nonzeros. The fixed boundary conditions~\eqref{eq:reform_boundary_row}--\eqref{eq:reform_boundary_col} contribute at most $O(N)$ nonzeros if implemented as explicit rows, and the equations~\eqref{eq:c_cover} contribute $\Theta(N^2)$ nonzeros.

Consequently, the total number of nonzero coefficients in the sparse coverage representation is
\[
\Theta(N^3)+\Theta(N^2)+\Theta(N^2)+O(N)+\Theta(N^2)=\Theta(N^3).
\]
The computational gain therefore comes from changing how the same coverage relation is represented: a placement variable for an $h\times h$ square no longer appears in all $h^2$ cell-cover constraints, but only in the four corner-update coefficients of~\eqref{eq:reform_d}; the global coverage is then propagated by the structured prefix equations~\eqref{eq:reform_prefix}.
\end{proof}

The main advantage of the reformulation is the reduced matrix density, which can make repeated LP solves and related linear-algebra operations cheaper. This benefit is especially relevant for large values of $N$, where large squares create many dense placement-to-cell incidences in the original model. The cost is that the model introduces auxiliary variables and structured equality constraints. Consequently, the row and column counts are not necessarily smaller, and branch-and-bound node counts need not decrease monotonically. The reformulation should therefore be viewed as a sparse extended formulation: it preserves the original placement decisions while changing the linear-algebra structure exposed to the solver.

\subsection{Inventory-Balance Interpretation}
\label{sec:inventory_interpretation}

The sparsity reduction described above is closely related to inventory-balance modeling \cite{pochet2006}. In a one-dimensional inventory model, a replenishment decision affects inventory levels over many future periods. A dense formulation could link the replenishment decision directly to each future period in which its effect is present. The standard inventory-balance formulation avoids this by recording only local inflows and outflows and then propagating their effect through recursive balance equations.

For example, if $y_t$ denotes a local replenishment or net transaction in period $t$, and $I_t$ denotes the resulting inventory level, the balance equations take the form
\begin{equation}
I_t = I_{t-1}+y_t,\qquad t=1,\ldots,T,\quad I_0=0. \label{eq:inventory_balance}
\end{equation}
Equivalently,
\[
I_t=\sum_{\tau=1}^{t}y_\tau.
\]
Thus the model does not need to connect a transaction at time $\tau$ explicitly to every later inventory level $I_t$ with $t\ge \tau$. The dense cumulative relation is represented by local balance equations.

CD-MILP applies the same principle in two spatial dimensions. The difference array $d$ plays the role of the transaction sequence $y$, and the coverage array $c$ plays the role of the inventory level $I$. The prefix equations are precisely the two-dimensional analogue of~\eqref{eq:inventory_balance}:
\begin{equation}
c_{u,v}=c_{u-1,v}+c_{u,v-1}-c_{u-1,v-1}+d_{u,v}. \label{eq:inventory_2d_balance}
\end{equation}
Equivalently, the reconstructed coverage is the cumulative sum
\begin{equation}
c_{u,v}=\sum_{p=1}^{u}\sum_{q=1}^{v} d_{p,q}. \label{eq:inventory_2d_prefix}
\end{equation}
A selected square can therefore be viewed as a local two-dimensional transaction. For a placement $(i,j,h)$, its contribution to $d$ is
\[
\Delta d_{p,q}
=
\mathbf{1}_{(p,q)=(i,j)}
-\mathbf{1}_{(p,q)=(i+h,j)}
-\mathbf{1}_{(p,q)=(i,j+h)}
+\mathbf{1}_{(p,q)=(i+h,j+h)}.
\]
Taking the cumulative sum in~\eqref{eq:inventory_2d_prefix} converts these four local transactions into the indicator of the square:
\[
\sum_{p=1}^{u}\sum_{q=1}^{v}\Delta d_{p,q}
=
\mathbf{1}_{\{i\le u\le i+h-1,\; j\le v\le j+h-1\}}.
\]
The original exact-cover model records this same effect by connecting the placement variable to all cells in the region. CD-MILP records only local transactions at the boundary: coverage is activated at one corner, canceled after crossing each of the two boundaries, and corrected by the opposite-corner inclusion-exclusion term.

This interpretation clarifies why the method removes nonzeros without changing the feasible set in the original placement variables. The dense relation ``a placement covers many cells'' is replaced by a sparse local-update system plus cumulative reconstruction. The same idea can be extended to higher-dimensional axis-aligned covering models.

\section{Symmetry Analysis and Handling}
\label{sec:symmetry}

\subsection[D4 Symmetry and Placement Orbits]{$D_4$ Symmetry and Placement Orbits}

Beyond matrix density, the square tiling problem exhibits strong symmetry. The most prominent symmetries are induced by the dihedral group $D_4$, which consists of the four rotations and four reflections of the square domain. These transformations preserve adjacency, side lengths, coverage, and the number of selected squares. Hence, if a placement vector $x$ is feasible, then every image of $x$ under a transformation in $D_4$ is also feasible and has the same objective value.

Figure~\ref{fig:d4_transformations} illustrates the eight transformations in $D_4$: the identity, three rotations, and four reflections. The numbers in the tiles denote side lengths. Each transformation preserves feasibility and objective value, while permuting the four corner-incident side lengths that define the corner signature.

These symmetries imply that each feasible solution belongs to an equivalence class of up to eight symmetric configurations, leading to redundant exploration in branch-and-bound search \cite{kaibel2008,margot2010}. In an unmodified MILP formulation, the solver may discover several branches that differ only by a rotation or reflection of the same geometric pattern. Such branches have identical objective values and essentially identical local LP structure, but they are not automatically recognized as equivalent by a general-purpose solver.

\begin{figure}[htbp]
    \centering
    \includegraphics[width=\linewidth]{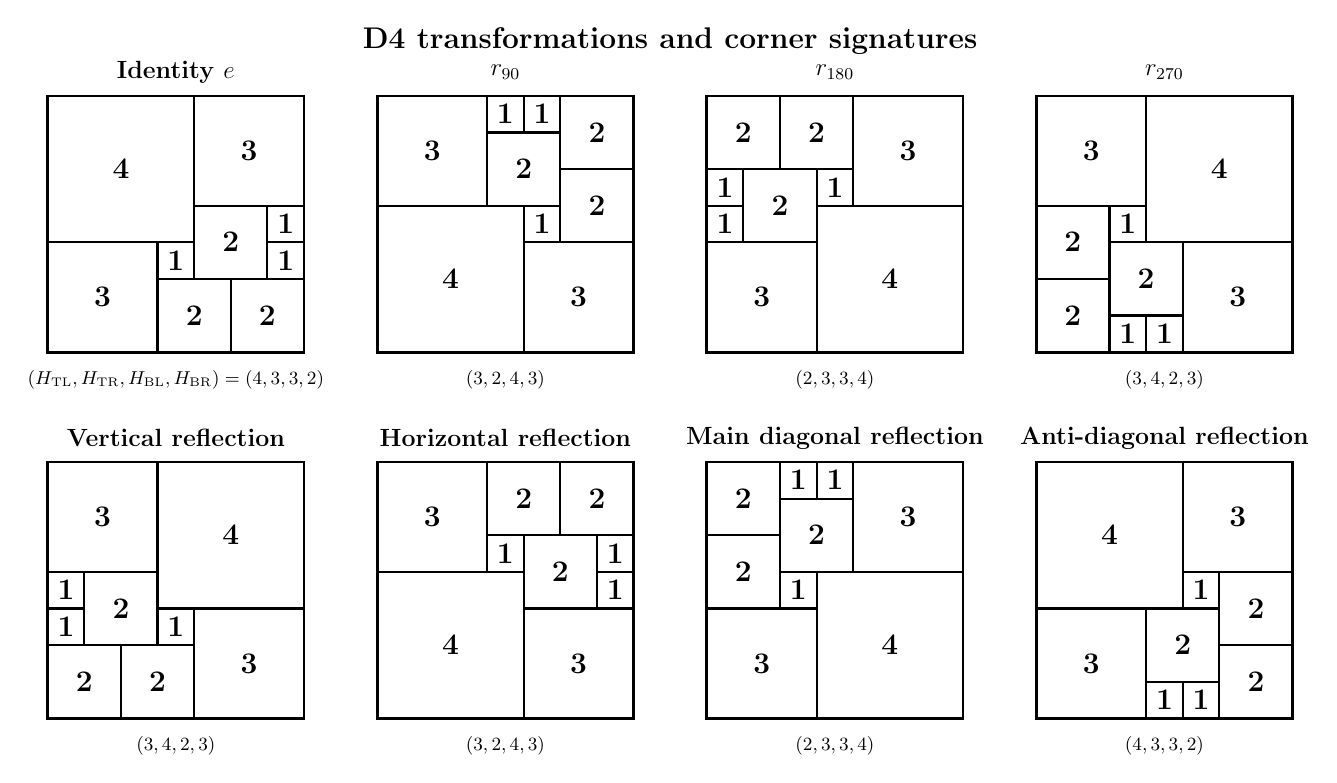}
    \caption{$D_4$ transformations and corner signatures.}
    \label{fig:d4_transformations}
\end{figure}

The symmetry acts directly on placement variables. For example, a square of side length $h$ with top-left corner $(i,j)$ is mapped by a $90^\circ$ clockwise rotation to a square of the same side length whose top-left corner is $(j,N-i-h+2)$. Reflections give analogous affine transformations of the placement coordinates. The important consequence is that the symmetry does not change the set of admissible side lengths or the objective coefficient of any selected square; it only relabels the placement variables. Therefore symmetry can be handled by adding symmetry-breaking constraints that choose one representative orientation from each orbit of the $D_4$ action.

\subsection{Corner-ordering symmetry breaking}

To reduce redundancy induced by these symmetries, we introduce a lightweight tie-breaking rule based on corner-incidence measures. Define the following sums representing the side lengths of squares that occupy the four corners of the $N\times N$ domain:
\begin{align}
H^{\mathrm{TL}} &= \sum_{h=1}^{N-1} h\,x_{1,1,h}, \notag \\
H^{\mathrm{TR}} &= \sum_{h=1}^{N-1} h\,x_{1,N-h+1,h}, \notag \\
H^{\mathrm{BL}} &= \sum_{h=1}^{N-1} h\,x_{N-h+1,1,h}, \notag \\
H^{\mathrm{BR}} &= \sum_{h=1}^{N-1} h\,x_{N-h+1,N-h+1,h} \label{eq:corner_def}
\end{align}
Because every corner cell must be covered exactly once, each quantity in \eqref{eq:corner_def} is the side length of the unique square incident to that corner. These four numbers provide a compact signature of the orientation of a tiling. Under the action of $D_4$, they are permuted in the same way as the four corners of the board.

We use this signature to select a canonical orientation. The first three inequalities require the top-left corner to have a side length at least as large as the side lengths at the other three corners. Since any corner can be mapped to the top-left corner by some element of $D_4$, at least one symmetric image of every feasible tiling satisfies these inequalities. After a corner with the largest incident-square side length has been mapped to the top-left, there remains a reflection across the main diagonal that fixes the top-left and bottom-right corners while swapping the top-right and bottom-left corners. The final inequality selects one of these two remaining orientations:
\begin{equation}
H^{\mathrm{TL}} \ge H^{\mathrm{TR}},\; H^{\mathrm{TL}} \ge H^{\mathrm{BL}},\; H^{\mathrm{TL}} \ge H^{\mathrm{BR}},\; H^{\mathrm{TR}} \ge H^{\mathrm{BL}} \label{eq:corner_ineq}
\end{equation}

\begin{proposition}[Validity of corner ordering]
\label{prop:corner_ordering}
For every feasible placement vector $x$, its orbit under the $D_4$ action contains at least one feasible placement vector satisfying the corner-ordering inequalities~\eqref{eq:corner_ineq}. Consequently, adding these inequalities to either the exact-cover formulation or CD-MILP does not change the optimal objective value.
\end{proposition}

\begin{proof}
The transformations in $D_4$ preserve feasibility and the objective value, and they only permute the four corner quantities in~\eqref{eq:corner_def}. Consider any feasible tiling and choose a corner whose incident square has maximum side length among the four corners. Some transformation in $D_4$ maps this corner to the top-left corner. In the transformed tiling, $H^{\mathrm{TL}}$ is therefore at least as large as $H^{\mathrm{TR}}$, $H^{\mathrm{BL}}$, and $H^{\mathrm{BR}}$.

It remains only to order the top-right and bottom-left corner quantities. If $H^{\mathrm{TR}}\ge H^{\mathrm{BL}}$ already holds, no further transformation is needed. Otherwise, reflect the tiling across the main diagonal. This reflection fixes the top-left and bottom-right corners and swaps the top-right and bottom-left corners, so $H^{\mathrm{TL}}$ remains the largest of the four corner quantities and the transformed tiling satisfies $H^{\mathrm{TR}}\ge H^{\mathrm{BL}}$. Thus every $D_4$ orbit contains at least one representative satisfying~\eqref{eq:corner_ineq}. Since at least one representative of every optimal orbit remains feasible, the optimal objective value is unchanged.
\end{proof}

These constraints are deliberately lightweight. They involve only placement variables that touch the four corners, so the number of additional nonzero coefficients is $O(N)$ and is negligible compared with either the dense exact-cover matrix or the CD-MILP formulation. They can be added directly to the baseline model as well as to the reformulated model because they are expressed entirely in the original placement variables.

The constraints should be interpreted as partial symmetry-breaking constraints. If the corner signature contains ties, some residual rotations or reflections may remain. Nevertheless, Proposition~\ref{prop:corner_ordering} shows that the inequalities are valid for optimization because they keep at least one representative from every $D_4$ orbit.

\section{Computational Results}
\label{sec:num}

This section evaluates the computational effect of CD-MILP and the corner-ordering symmetry-breaking constraints. All reported runs use the same instance set, thread count, and time limit for a given solver, so that the comparison focuses on the formulation rather than on changes in the experimental protocol.

All experiments were conducted using \textsc{Gurobi} and \textsc{COPT} on the same hardware platform: Ubuntu 24.04.2 LTS with an AMD Ryzen Threadripper 7970X 32-core processor and 512 GB RAM. The \textsc{Gurobi} runs used \textsc{Gurobi} Optimizer 13.0.2, and the \textsc{COPT} runs used Cardinal Optimizer v8.0.4 on Linux (Ubuntu 24.04.2 LTS, x86\_64). Unless otherwise stated, each solver was executed with eight threads (\texttt{Threads=8}), the default presolve level enabled explicitly (\texttt{Presolve=1}), a time limit of 7200 seconds, and zero relative MIP optimality gap tolerance (\texttt{MIPGap=0}). For instances that reached the time limit, we report the final optimality gap returned by the solver.

We consider square tiling instances with side lengths
\[
N \in \{31,37,41,43,47,53,59,61,67\}.
\]
The instance set focuses on prime benchmark sizes, including larger values of $N$, in order to test whether the reduction in nonzero coefficients translates into improved scalability and to avoid decompositions induced by nontrivial divisors of the side length.

We evaluate four MILP formulations:
\begin{itemize}
    \item Exact-Cover MILP: cell-level exact-cover formulation (Section~\ref{sec:orig});
    \item Corner-Difference MILP: corner-difference formulation (Section~\ref{sec:2ddiff});
    \item Exact-Cover MILP with Corner Ordering: exact-cover formulation with corner-ordering constraints (Section~\ref{sec:symmetry});
    \item Corner-Difference MILP with Corner Ordering: Corner-Difference MILP with the corner-ordering constraints \eqref{eq:corner_ineq}.
\end{itemize}

\subsection{Structural Sparsity and Model-Size Reduction}
\begin{table}[!htbp]
\centering
\caption{Model-size statistics.}
\label{tab:model_size}
\setlength{\tabcolsep}{3pt}
\begin{tabular}{llrrrrrrr}
\toprule
Solver & Instance & \multicolumn{3}{c}{Exact-Cover MILP}
         & \multicolumn{3}{c}{\shortstack{Corner-Difference MILP\\with Corner Ordering}}
         & NZ reduction \\
\cmidrule(lr){3-5} \cmidrule(lr){6-8}
 & & Rows & Cols & NZ
 & Rows & Cols & NZ
 & factor \\
\midrule
\textsc{Gurobi} & $31$ & 841 & 10,171 & 145,730 & 965 & 10,415 & 39,947 & 3.65 \\
 & $37$ & 1,225 & 17,280 & 1,345,798 & 1,373 & 17,574 & 67,809 & 19.85 \\
 & $41$ & 1,521 & 23,494 & 2,871,126 & 1,685 & 23,820 & 92,197 & 31.14 \\
 & $43$ & 1,681 & 27,091 & 4,736,302 & 1,853 & 27,433 & 106,327 & 44.54 \\
 & $47$ & 2,025 & 35,343 & 8,292,331 & 2,213 & 35,719 & 138,779 & 59.75 \\
 & $53$ & 2,601 & 50,614 & 15,016,046 & 2,813 & 51,038 & 198,897 & 75.50 \\
 & $59$ & 3,249 & 69,737 & 25,519,545 & 3,488 & 70,212 & 274,059 & 93.11 \\
 & $61$ & 3,481 & 77,042 & 30,094,818 & 3,728 & 77,533 & 302,869 & 99.36 \\
 & $67$ & 4,225 & 101,973 & 47,876,237 & 4,496 & 102,512 & 401,267 & 119.31 \\
\midrule
\textsc{COPT} & $31$ & 837 & 7,716 & 796,048 & 1,012 & 10,417 & 41,126 & 19.36 \\
 & $37$ & 1,221 & 13,685 & 1,988,129 & 1,432 & 17,579 & 69,657 & 28.54 \\
 & $41$ & 1,517 & 19,019 & 3,375,251 & 1,769 & 23,853 & 95,284 & 35.42 \\
 & $43$ & 1,677 & 22,140 & 4,312,048 & 1,941 & 27,471 & 109,753 & 39.29 \\
 & $47$ & 2,021 & 29,370 & 6,806,642 & 2,309 & 35,776 & 142,976 & 47.61 \\
 & $53$ & 2,597 & 42,925 & 12,587,561 & 2,921 & 51,116 & 204,333 & 61.60 \\
 & $59$ & 3,245 & 60,116 & 21,758,264 & 3,605 & 70,299 & 281,053 & 77.42 \\
 & $61$ & 3,477 & 66,729 & 25,786,477 & 3,849 & 77,624 & 310,349 & 83.09 \\
 & $67$ & 4,221 & 89,440 & 41,565,468 & 4,629 & 102,615 & 410,301 & 101.30 \\
\bottomrule
\end{tabular}

\end{table}

Table~\ref{tab:model_size} reports the model-size statistics for the Exact-Cover MILP and the Corner-Difference MILP with Corner Ordering. These statistics isolate the structural effect of the corner-difference representation before branch-and-bound performance is considered. NZ denotes the number of nonzero coefficients in the presolved constraint matrix. The reduction factor is computed as the Exact-Cover MILP NZ divided by the Corner-Difference MILP with Corner Ordering NZ.

Corner-Difference MILP with Corner Ordering has slightly more rows and columns, because it introduces the auxiliary difference and prefix variables. However, the number of nonzero coefficients is reduced by one to two orders of magnitude. For $N = 67$, Corner-Difference MILP with Corner Ordering achieves approximately a 119-fold reduction in nonzeros relative to the exact-cover baseline with \textsc{Gurobi}, and approximately a 101-fold reduction with \textsc{COPT}.

This confirms the mechanism described in Proposition~\ref{prop:nnz}. Corner-Difference MILP with Corner Ordering is not primarily smaller in row or column count; rather, it replaces dense placement-to-cell incidence by sparse local updates and prefix reconstruction. The computational results suggest that this structural change is valuable for modern MILP solvers, especially on the larger instances where repeated LP processing dominates the solution effort.

\subsection{Overall effect of the formulations}

Table~\ref{tab:performance_summary} summarizes the aggregate behavior of the four formulations. ``Solved'' is the number of instances solved to optimality within the time limit, and all averages are computed over the nine tested instances. For unsolved instances, 7200 seconds is used when computing average runtime and the aggregate LP-processing rate. Avg. NZ is the average number of nonzero coefficients in the presolved constraint matrix, and LP iter./s is cumulative LP iterations divided by total runtime. Detailed ablation results and model-size statistics for the intermediate Corner-Difference MILP and Exact-Cover MILP with Corner Ordering formulations are provided in Tables~\ref{tab:app_reform_cut} and~\ref{tab:app_model_size_reform_cut}.

\begin{table}[htbp]
\centering
\caption{Aggregate computational performance.}
\label{tab:performance_summary}
\setlength{\tabcolsep}{3pt}
\begin{tabular}{@{}lrrrrr@{}}
\toprule
Formulation & Solved & \shortstack{Avg.\\$T$ (s)} & \shortstack{Avg.\\nodes}
            & \shortstack{Avg.\\NZ} & \shortstack{LP\\iter./s} \\
\midrule
\multicolumn{6}{l}{\textsc{Gurobi}} \\
Exact-Cover MILP & 4 & 4,150.59 & 2,754.33 & 1.51e7 & 83.79 \\
Corner-Difference MILP & 6 & 2,857.68 & 4,236.89 & 1.80e5 & 271.00 \\
Exact-Cover MILP with Corner Ordering & 6 & 3,074.84 & 1,630.11 & 1.52e7 & 88.17 \\
Corner-Difference MILP with Corner Ordering & 8 & 1,876.12 & 2,411.44 & 1.80e5 & 232.16 \\
\midrule
\multicolumn{6}{l}{\textsc{COPT}} \\
Exact-Cover MILP & 5 & 4,485.15 & 3,888.67 & 1.32e7 & 106.03 \\
Corner-Difference MILP & 5 & 4,278.26 & 5,207.44 & 1.85e5 & 252.44 \\
Exact-Cover MILP with Corner Ordering & 5 & 3,785.55 & 1,095.00 & 1.35e7 & 47.55 \\
Corner-Difference MILP with Corner Ordering & 7 & 2,475.91 & 2,826.44 & 1.85e5 & 162.90 \\
\bottomrule
\end{tabular}
\end{table}

The first clear effect is the reduction in the number of nonzero coefficients. For both solvers, the Corner-Difference MILP have average NZ around $10^5$, whereas the corresponding cell-incidence formulations have average NZ around $10^7$. This confirms that the sparsity gain appears in the solver-reported model statistics, not only in the symbolic formulation.

The iteration rates in Table~\ref{tab:performance_summary} should be read together with the model-size statistics: for both solvers, using Corner-Difference MILP increases the number of LP iterations processed per second, both without and with the corner-ordering constraints. This is consistent with solving LPs on much sparser matrices. The corner-ordering constraints have a different profile: they do not materially reduce NZ, and their effect on the iteration rate is secondary and mixed; their main contribution is reflected in the reduced node counts and runtimes.

The second effect is algorithmic. The Corner-Difference MILP alone substantially reduces the \textsc{Gurobi} average runtime and solves two additional instances relative to the baseline. For \textsc{COPT}, the Corner-Difference MILP alone gives a smaller improvement in solved count, but still reduces the reported nonzero count by nearly two orders of magnitude.

The combined formulation gives the best overall result. It solves eight of the nine instances with \textsc{Gurobi} and seven with \textsc{COPT}, compared with four and five, respectively, for the exact-cover baseline. It also has the smallest average runtime for both solvers. The node counts are not uniformly smaller. This should not be read as a contradiction to Proposition~\ref{prop:lp}: projected LP equivalence only compares the feasible placement vectors in the continuous relaxation, and does not imply that a solver will process identical extended LPs, generate the same cuts, or produce the same branch-and-bound tree. Thus node count is a formulation- and solver-dependent procedural statistic, while solved count, runtime, final gaps, and NZ provide a more stable summary of the observed computational effect.

The ablation results in Appendix~\ref{app:ablation} also show that the corner-ordering constraints should not be interpreted as a standalone primal-search improvement. For example, on the \textsc{Gurobi} run for $N=67$, Exact-Cover MILP with Corner Ordering terminates with a poor incumbent despite processing very few nodes. This anomaly suggests that partial symmetry breaking can alter solver heuristics and primal search in instance-dependent ways, so the main empirical conclusion is the complementarity of Corner-Difference MILP and corner ordering rather than a uniform benefit from either component alone.

\subsection{Per-instance comparison by solver}

Table~\ref{tab:solver_detail} compares the Exact-Cover MILP and the Corner-Difference MILP with Corner Ordering instance by instance. UB is the incumbent objective value, BestBd is the solver-reported final best bound, $G$ is the final relative MIP optimality gap, and $T$ is the wall-clock runtime. This view is useful because aggregate averages hide where the improvement occurs. The small and medium instances are solved by both formulations, while the larger instances expose the difference in scalability.

\begin{table}[!htbp]
\centering
\caption{Per-instance computational performance.}
\label{tab:solver_detail}
\setlength{\tabcolsep}{1.6pt}
\begin{tabular}{llrrrrr rrrrr}
\toprule
Solver & Instance & \multicolumn{5}{c}{Exact-Cover MILP}
         & \multicolumn{5}{c}{\shortstack{Corner-Difference MILP\\with Corner Ordering}} \\
\cmidrule(lr){3-7} \cmidrule(lr){8-12}
 & & UB & BestBd & $T$ (s) & Nodes & $G$ (\%)
 & UB & BestBd & $T$ (s) & Nodes & $G$ (\%) \\
\midrule
\textsc{Gurobi} & $31$ & 15 & 15 & 32.00 & 1,522 & 0.00 & 15 & 15 & 14.25 & 608 & 0.00 \\
 & $37$ & 15 & 15 & 184.39 & 2,258 & 0.00 & 15 & 15 & 49.84 & 429 & 0.00 \\
 & $41$ & 15 & 15 & 604.26 & 4,252 & 0.00 & 15 & 15 & 113.81 & 1,320 & 0.00 \\
 & $43$ & 16 & 16 & 534.64 & 4,360 & 0.00 & 16 & 16 & 186.30 & 1,389 & 0.00 \\
 & $47$ & 16 & 13 & 7,200.05 & 8,704 & 18.75 & 16 & 16 & 292.90 & 1,510 & 0.00 \\
 & $53$ & 16 & 11 & 7,200.10 & 2,555 & 31.25 & 16 & 16 & 1,384.06 & 4,488 & 0.00 \\
 & $59$ & 17 & 11 & 7,200.61 & 748 & 35.29 & 17 & 17 & 2,918.39 & 3,981 & 0.00 \\
 & $61$ & 17 & 10 & 7,201.56 & 364 & 41.18 & 17 & 17 & 4,725.54 & 4,680 & 0.00 \\
 & $67$ & 18 & 10 & 7,201.97 & 26 & 44.44 & 17 & 12 & 7,200.02 & 3,298 & 29.41 \\
\midrule
\textsc{COPT} & $31$ & 15 & 15 & 95.19 & 3,822 & 0.00 & 15 & 15 & 47.82 & 313 & 0.00 \\
 & $37$ & 15 & 15 & 318.85 & 3,406 & 0.00 & 15 & 15 & 152.63 & 585 & 0.00 \\
 & $41$ & 15 & 15 & 856.41 & 2,589 & 0.00 & 15 & 15 & 355.49 & 815 & 0.00 \\
 & $43$ & 16 & 16 & 3,251.88 & 10,851 & 0.00 & 16 & 16 & 1,277.24 & 2,965 & 0.00 \\
 & $47$ & 16 & 16 & 7,043.99 & 10,272 & 0.00 & 16 & 16 & 370.82 & 1,536 & 0.00 \\
 & $53$ & 16 & 12.80 & 7,200.48 & 2,561 & 19.98 & 16 & 16 & 1,143.51 & 3,083 & 0.00 \\
 & $59$ & 17 & 9.17 & 7,200.82 & 734 & 46.04 & 17 & 17 & 4,535.69 & 5,520 & 0.00 \\
 & $61$ & 17 & 9.46 & 7,201.02 & 562 & 44.37 & 17 & 14.78 & 7,200.04 & 9,606 & 13.06 \\
 & $67$ & 18 & 9.13 & 7,202.01 & 201 & 49.29 & 17 & 12.36 & 7,200.03 & 1,015 & 27.31 \\
\bottomrule
\end{tabular}

\end{table}

For \textsc{Gurobi}, the exact-cover baseline reaches the time limit for all tested instances with $N\ge47$, while Corner-Difference MILP with Corner Ordering proves optimality for all instances up to $N=61$. On $N=67$, neither formulation proves optimality within the time limit, but Corner-Difference MILP with Corner Ordering finds a better incumbent value and leaves a smaller final gap. The improvement is especially pronounced for $N=47$ through $N=61$, where Corner-Difference MILP with Corner Ordering closes instances that remain open under the baseline.

For \textsc{COPT}, the same qualitative pattern appears with a slightly different cutoff. Corner-Difference MILP with Corner Ordering solves all instances up to $N=59$, whereas the exact-cover baseline leaves positive gaps for $N=53,59,61,$ and $67$. On the two hardest instances, $N=61$ and $N=67$, Corner-Difference MILP with Corner Ordering still times out, but its final gaps are substantially smaller than those of the baseline.

\FloatBarrier
\section{Conclusion}
\label{sec:conclu}
In this paper, we studied MILP formulations for the square tiling problem and proposed CD-MILP, an exact sparse reformulation that can be combined with corner-ordering symmetry-breaking constraints. The main modeling contribution is to replace the dense cell-incidence structure of the classical exact-cover formulation with local four-corner updates and prefix-sum reconstruction. This reformulation preserves the original tiling decisions while reducing the number of placement-related nonzero coefficients from $\Theta(N^5)$ to $\Theta(N^3)$.

The proposed model also clarifies why the reduction is possible. The corner-difference representation can be interpreted as a two-dimensional extension of inventory balance modeling: instead of linking a decision to every location affected by it, the model records local net changes and reconstructs cumulative coverage through conservation-type equations. This viewpoint gives an intuitive explanation for the sparsity improvement and may be useful for related geometric covering and packing problems.

The corner-based symmetry-breaking inequalities address a different source of difficulty. By imposing a canonical ordering on the square sizes incident to the four corners, the model removes some redundant representatives generated by the dihedral symmetries of the square while preserving at least one representative from each symmetry orbit. These inequalities are linear and lightweight, and they can be added directly to either the original formulation or the sparse reformulation.

The computational results show that these two mechanisms are complementary. Across both \textsc{Gurobi} and \textsc{COPT}, Corner-Difference MILP with Corner Ordering solves more instances within the time limit than the exact-cover baseline and gives smaller remaining gaps on the hardest tested instances. The model-size comparison confirms the intended sparsity effect: the number of rows and columns changes only moderately, while the number of nonzero coefficients is reduced by one to two orders of magnitude on the larger instances.

These results suggest that sparse cumulative reformulations are a useful modeling tool for geometric covering problems whose natural exact-cover formulations contain dense placement-to-cell incidence matrices. Future work could extend the approach to other tiling and packing variants, incorporate stronger problem-specific valid inequalities, and study how solver presolve routines interact with difference-based formulations.

\clearpage
\appendix

\section{Ablation Details}
\label{app:ablation}

This appendix provides detailed ablation results and model-size statistics for the Corner-Difference MILP and the Exact-Cover MILP with Corner Ordering. The former excludes symmetry-breaking constraints, whereas the latter augments the exact-cover formulation with corner-ordering constraints.

\begin{table}[!htbp]
\centering
\caption{Ablation performance.}
\label{tab:app_reform_cut}
\setlength{\tabcolsep}{1.6pt}
\begin{tabular}{llrrrrr rrrrr}
\toprule
Solver & Instance & \multicolumn{5}{c}{Corner-Difference MILP}
       & \multicolumn{5}{c}{\shortstack{Exact-Cover MILP\\with Corner Ordering}} \\
\cmidrule(lr){3-7} \cmidrule(lr){8-12}
 & & UB & BestBd & $T$ (s) & Nodes & $G$ (\%)
 & UB & BestBd & $T$ (s) & Nodes & $G$ (\%) \\
\midrule
\textsc{Gurobi}
 & $31$ & 15 & 15 & 21.67 & 1,317 & 0.00 & 15 & 15 & 42.78 & 384 & 0.00 \\
 & $37$ & 15 & 15 & 66.45 & 1,370 & 0.00 & 15 & 15 & 288.51 & 1,715 & 0.00 \\
 & $41$ & 15 & 15 & 107.54 & 1,072 & 0.00 & 15 & 15 & 791.49 & 3,109 & 0.00 \\
 & $43$ & 16 & 16 & 241.78 & 3,070 & 0.00 & 16 & 16 & 602.02 & 1,527 & 0.00 \\
 & $47$ & 16 & 16 & 590.65 & 3,161 & 0.00 & 16 & 16 & 857.33 & 1,169 & 0.00 \\
 & $53$ & 16 & 16 & 3,053.13 & 9,627 & 0.00 & 16 & 16 & 3,491.41 & 2,180 & 0.00 \\
 & $59$ & 17 & 13 & 7,200.06 & 8,270 & 23.53 & 17 & 15 & 7,200.15 & 2,499 & 11.76 \\
 & $61$ & 17 & 13 & 7,200.01 & 7,842 & 23.53 & 17 & 14 & 7,200.07 & 2,083 & 17.65 \\
 & $67$ & 17 & 11 & 7,200.01 & 2,403 & 35.29 & 40 & 10 & 7,200.91 & 5 & 75.00 \\
\midrule
\textsc{COPT}
 & $31$ & 15 & 15 & 59.17 & 989 & 0.00 & 15 & 15 & 30.99 & 346 & 0.00 \\
 & $37$ & 15 & 15 & 363.45 & 2,397 & 0.00 & 15 & 15 & 227.56 & 775 & 0.00 \\
 & $41$ & 15 & 15 & 986.59 & 4,383 & 0.00 & 15 & 15 & 1,054.46 & 1,638 & 0.00 \\
 & $43$ & 16 & 16 & 3,462.51 & 13,809 & 0.00 & 16 & 16 & 1,089.52 & 1,750 & 0.00 \\
 & $47$ & 16 & 16 & 4,832.59 & 12,157 & 0.00 & 16 & 16 & 2,867.39 & 1,660 & 0.00 \\
 & $53$ & 16 & 10.06 & 7,200.02 & 7,818 & 37.13 & 17 & 9.87 & 7,200.49 & 2,316 & 41.91 \\
 & $59$ & 17 & 9.62 & 7,200.04 & 2,712 & 43.42 & 17 & 9.25 & 7,200.84 & 675 & 45.61 \\
 & $61$ & 17 & 9.51 & 7,200.15 & 1,523 & 44.03 & 17 & 9.24 & 7,201.29 & 595 & 45.62 \\
 & $67$ & 17 & 9.42 & 7,200.09 & 1,079 & 44.59 & 18 & 9.17 & 7,202.04 & 100 & 49.03 \\
\bottomrule
\end{tabular}
\end{table}
\FloatBarrier

Table~\ref{tab:app_reform_cut} uses the performance metrics defined for Table~\ref{tab:solver_detail}, while Table~\ref{tab:app_model_size_reform_cut} uses the same presolved model-size metrics as Table~\ref{tab:model_size}.

\begin{table}[htbp]
\centering
\caption{Ablation model-size statistics.}
\label{tab:app_model_size_reform_cut}
\setlength{\tabcolsep}{2pt}
\begin{tabular}{llrrr rrr}
\toprule
Solver & Instance & \multicolumn{3}{c}{Corner-Difference MILP}
         & \multicolumn{3}{c}{\shortstack{Exact-Cover MILP\\with Corner Ordering}} \\
\cmidrule(lr){3-5} \cmidrule(lr){6-8}
 & & Rows & Cols & NZ & Rows & Cols & NZ \\
\midrule
\textsc{Gurobi} & $31$ & 961 & 10,415 & 39,710 & 845 & 10,252 & 155,276 \\
 & $37$ & 1,369 & 17,574 & 67,524 & 1,229 & 17,379 & 1,398,685 \\
 & $41$ & 1,681 & 23,820 & 91,880 & 1,525 & 23,602 & 3,035,001 \\
 & $43$ & 1,849 & 27,433 & 105,994 & 1,685 & 27,205 & 4,713,735 \\
 & $47$ & 2,209 & 35,719 & 138,414 & 2,029 & 35,469 & 8,350,666 \\
 & $53$ & 2,809 & 51,038 & 198,484 & 2,605 & 50,758 & 15,105,626 \\
 & $59$ & 3,481 & 70,209 & 273,818 & 3,256 & 69,917 & 25,690,737 \\
 & $61$ & 3,721 & 77,530 & 302,620 & 3,488 & 77,228 & 30,285,173 \\
 & $67$ & 4,489 & 102,509 & 400,994 & 4,232 & 102,177 & 48,132,553 \\
\midrule
\textsc{COPT} & $31$ & 1,024 & 10,478 & 41,723 & 847 & 7,919 & 843,277 \\
 & $37$ & 1,444 & 17,649 & 70,371 & 1,231 & 13,938 & 2,077,634 \\
 & $41$ & 1,764 & 23,903 & 95,363 & 1,527 & 19,301 & 3,499,033 \\
 & $43$ & 1,936 & 27,520 & 109,819 & 1,687 & 22,438 & 4,457,408 \\
 & $47$ & 2,304 & 35,814 & 142,971 & 2,031 & 29,700 & 7,002,486 \\
 & $53$ & 2,916 & 51,145 & 204,259 & 2,607 & 43,300 & 12,874,075 \\
 & $59$ & 3,600 & 70,328 & 280,955 & 3,255 & 60,536 & 22,159,853 \\
 & $61$ & 3,844 & 77,653 & 310,243 & 3,487 & 67,165 & 26,234,688 \\
 & $67$ & 4,624 & 102,644 & 410,171 & 4,231 & 89,924 & 42,175,081 \\
\bottomrule
\end{tabular}
\end{table}

\FloatBarrier

The following appendices list representative optimal tilings for $N=41$ and $N=47$. These two instances are included to make the displayed tiling figures reproducible; the remaining benchmark instances are used only in the computational comparison tables and are not listed as placement tables here.

\clearpage
\section[Optimal Solution for the 41 x 41 Instance]{Optimal Solution for the $41\times 41$ Instance}
\label{app:square41}

This appendix reports an optimal tiling for the $41\times 41$ instance. Figure~\ref{fig:square41} shows the tiling. Placement coordinates use the same one-based upper-left cell convention as the formulations in Sections~\ref{sec:orig}--\ref{sec:2ddiff}; if supplementary files use zero-based row and column coordinates, add one to each row and column coordinate to obtain the coordinates used here. Each row of Table~\ref{tab:placements41} lists a placement variable $x_{i,j,h}$, where $(i,j)$ is the upper-left cell coordinate and $h$ is the side length.

\begin{figure}[htbp]
    \centering
    \includegraphics[width=0.62\linewidth]{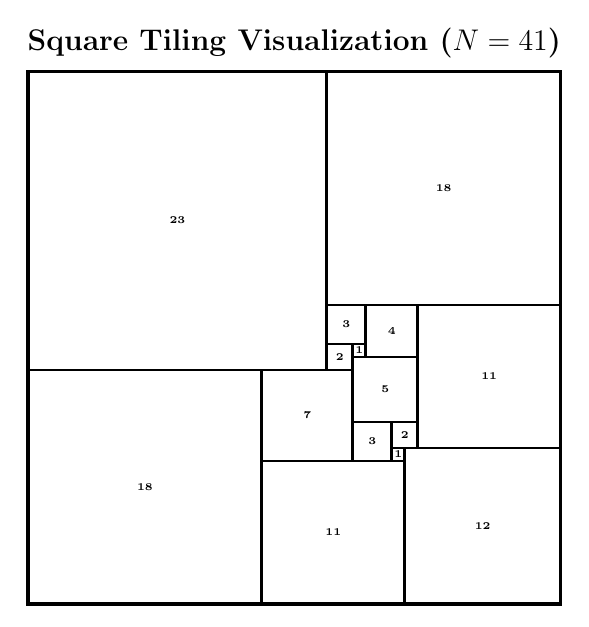}
\caption{Optimal tiling for $N=41$.}
    \label{fig:square41}
\end{figure}
\FloatBarrier

\begin{table}[htbp]
\centering
\caption{Placement coordinates for Figure~\ref{fig:square41}.}
\label{tab:placements41}
\begin{tabular}{rrr rrr}
\toprule
\multicolumn{3}{c}{Placement} & \multicolumn{3}{c}{Placement} \\
\cmidrule(lr){1-3} \cmidrule(lr){4-6}
Row $i$ & Column $j$ & Side $h$ & Row $i$ & Column $j$ & Side $h$ \\
\midrule
1  & 1  & 23 & 19 & 27 & 4 \\
1  & 24 & 18 & 19 & 24 & 3 \\
24 & 1  & 18 & 28 & 26 & 3 \\
19 & 31 & 11 & 22 & 24 & 2  \\
30 & 30 & 12 & 28 & 29 & 2 \\
31 & 19 & 11 & 22 & 26 & 1  \\
24 & 19 & 7  & 30 & 29 & 1   \\
23 & 26 & 5  &    &    &   \\
\bottomrule
\end{tabular}
\end{table}
\FloatBarrier

\clearpage

\section[Optimal Solution for the 47 x 47 Instance]{Optimal Solution for the $47\times 47$ Instance}
\label{app:square47}

This appendix reports an optimal tiling for the $47\times 47$ instance. Figure~\ref{fig:square47} shows the tiling. Placement coordinates use the same one-based upper-left cell convention as the formulations in Sections~\ref{sec:orig}--\ref{sec:2ddiff}; if supplementary files use zero-based row and column coordinates, add one to each row and column coordinate to obtain the coordinates used here. Each row of Table~\ref{tab:placements47} lists a placement variable $x_{i,j,h}$, where $(i,j)$ is the upper-left cell coordinate and $h$ is the side length.

\begin{figure}[htbp]
    \centering
    \includegraphics[width=0.62\linewidth]{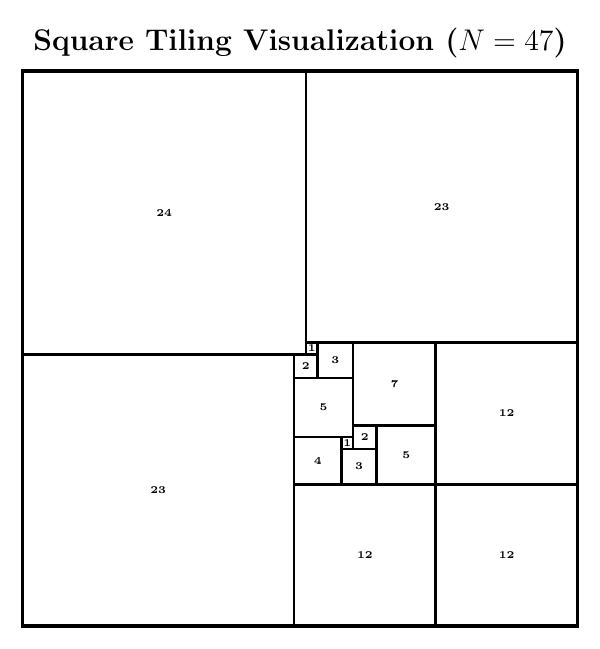}
\caption{Optimal tiling for $N=47$.}
    \label{fig:square47}
\end{figure}
\FloatBarrier

\begin{table}[htbp]
\centering
\caption{Placement coordinates for Figure~\ref{fig:square47}.}
\label{tab:placements47}
\begin{tabular}{rrr rrr}
\toprule
\multicolumn{3}{c}{Placement} & \multicolumn{3}{c}{Placement} \\
\cmidrule(lr){1-3} \cmidrule(lr){4-6}
Row $i$ & Column $j$ & Side $h$ & Row $i$ & Column $j$ & Side $h$ \\
\midrule
1  & 1  & 24 & 31 & 31 & 5 \\
1  & 25 & 23 & 32 & 24 & 4 \\
25 & 1  & 23 & 24 & 26 & 3\\
24 & 36 & 12 & 33 & 28 & 3 \\
36 & 36 & 12 & 25 & 24 & 2\\
36 & 24 & 12 & 31 & 29 & 2\\
24 & 29 & 7  & 24 & 25 & 1\\
27 & 24 & 5  & 32 & 28 & 1  \\

\bottomrule
\end{tabular}
\end{table}
\FloatBarrier

\bibliography{refs}

\FloatBarrier
\section*{Statements and Declarations}

\paragraph*{Funding}
The authors declare that no funds, grants, or other support were received during the preparation of this manuscript.

\paragraph*{Competing Interests}
The authors have no relevant financial or non-financial interests to disclose.

\paragraph*{Author Contributions}
Zhuo Yu: Software, computational experiments, and
writing--original draft. Yuan Wang: Methodology, supervision, and writing--review and editing. Zhuo Liu: computational experiments and data analysis. All authors read and approved the final manuscript.

\end{document}